\documentclass[11pt, a4paper]{article}
\usepackage{array,color,colortbl} 
\usepackage{graphics}
\usepackage[toc,page]{appendix}
\usepackage{float}
\usepackage{amsmath}
\usepackage{amsthm}
\usepackage{units}
\usepackage{float}%
\usepackage{makebox}%
\usepackage{relsize}%
\usepackage{empheq}%
\usepackage{booktabs} % for much better looking tables
\usepackage{array} % for better arrays (eg matrices) in maths
\usepackage{paralist} % very flexible & customisable lists (eg. enumerate/itemize, etc.)
\usepackage{verbatim} % adds environment for commenting out blocks of text & for better verbatim
\usepackage{subfig} % make it possible to include more than one captioned figure/table in a single float
\usepackage{mathtools}
\usepackage{blkarray}
\usepackage{lscape}
\usepackage[T1]{fontenc}
\usepackage[utf8]{inputenc}
\usepackage{authblk}
\usepackage{arydshln}
\usepackage{amssymb}
\usepackage{enumitem}
\usepackage[nobottomtitles*]{titlesec}

\newcommand{\UP}[2]{\makebox[0pt]{\smash{\raisebox{1.9em}{$\phantom{#2}#1$}}}#2}
\newcommand{\LF}[1]{\makebox[0pt]{$#1$\hspace{6.0em}}}

\theoremstyle{definition}
\newtheorem{lemma}{Lemma}

\newtheorem{example}{Example}
\newtheorem{theorem}{Theorem}
\newtheorem{corollary}{Corollary}

\title{Row-column factorial designs with multiple levels}

\author{Fahim Rahim and Nicholas J. Cavenagh}
\affil{Department of Mathematics and Statistics, \\
University of Waikato, Private Bag 3105, \\
Waikato Mail Centre 3240, New Zealand}
\begin{document}
	
\maketitle

\begin{abstract}
An  {\em $m\times n$ row-column factorial design} is an arrangement of the elements of a factorial design into a rectangular array. 
Such an array is used in experimental design, where the rows and columns can act as blocking factors. 
If for each row/column and vector position, each element has the same regularity,  
then all main effects can be estimated without confounding by the row and column blocking factors.  

Formally, for any integer $q$,  let $[q]=\{0,1,\dots ,q-1\}$. 
The $q^k$ (full) factorial design with replication $\alpha$ is the multi-set consisting of $\alpha$ occurrences of each element of $[q]^k$; we denote this by $\alpha\times [q]^k$.
A {\em regular $m\times n$ row-column factorial design} is an arrangement of the the elements of $\alpha \times [q]^k$ into an $m\times n$ array (which we say is of {\em type} $I_k(m,n;q)$)   
such that for each row (column) and fixed vector position $i\in [q]$, each element of $[q]$ occurs $n/q$ times (respectively, $m/q$ times).  
Let $m\leq n$. 
We show that an array of type $I_k(m,n;q)$ exists if and only if (a) $q|m$ and $q|n$; (b) $q^k|mn$; (c) $(k,q,m,n)\neq (2,6,6,6)$ and (d) if $(k,q,m)=(2,2,2)$ then $4$ divides $n$. 
  This extends the work of Godolphin (2019), who showed the above is true for the case $q=2$ when $m$ and $n$ are powers of $2$.  
  
  In the case $k=2$, the above implies necessary and sufficient conditions for the existence of a pair of mutually orthogonal frequency rectangles (or $F$-rectangles) whenever each symbol occurs the same number of times in a given row or column.  
\end{abstract}

{\bf Keywords}: Row-column factorial design, blocking factor, double confounding, frequency square {\em or} $F$-square, frequency rectangle {\em or} $F$-rectangle, MOFS.  
	
\section{Introduction}
For any integer $q$,  let $[q]=\{0,1,\dots ,q-1\}$.
Consider the following example of an experimental design from \cite{doi:10.1080/03610926.2015.1122062}. Suppose we wish to study the effects of four drugs on calves, where there are four breeds, four age groups and two levels for dosage for each type of drug. 
We could conduct 16 experiments based on the following row-column factorial design: 
 \begin{table}[H]
		\begin{center}
			\renewcommand{\arraystretch}{1.2}
			\begin{tabular}{cccccccccccc}
				1111&0100&0010&1001 \\
				0001&1010&1100&0111 \\
				1000&0011&0101&1110 \\
				0110&1101&1011&0000 \\
	
			\end{tabular}
		\caption{A regular row-column factorial design of type  $I_4(4,4;2)$.}
		\end{center} 
	\end{table}
	Here the rows and columns correspond to age groups and  breeds, respectively, of calves; with the binary vector in a cell indicating the dosage of each of the four drugs as one of two {\em levels}. 
	
In the above the $16$ vectors from $[2]^4$ are arranged in a  $4\times 4$ array, in such a way that for each row (column) and $i\in [4]$, the entries $0$ and $1$ each appear twice in 
position $i$ of a vector in that row (respectively, column). These properties of regularity in the experimental design allows unbiased estimation of the effects of each drug, together with the effects of breed and age group (which can be thought of as {\em blocking factors}), without any confounding between these main effects. 

 Formally,  the $q^k$ (full) factorial design with replication $\alpha$ is the multi-set consisting of $\alpha$ occurrences of each element of $[q]^k$; we denote this by $\alpha\times [q]^k$. An $m\times n$ {\em row-column factorial design} $q^k$ is any arrangement of the elements of $\alpha \times q^k$ into an $m\times n$ array. Necessarily, $q^k$ must divide $mn$. Without loss of generality, we always assume $m\leq n$.  
 We call such a design {\em regular} if for each row (column) and $i\in [q]$, each element of $[q]$ occurs $n/q$ times (respectively, $m/q$ times). Furthermore we denote the type of such an array to be $I_k(m,n;q)$, where regularity is always assumed to hold. 
Observe that regularity implies that $q$ divides both $m$ and $n$.   
The above example is thus a regular $4\times 4$ row-column factorial design $2^4$, or equivalently an array of  type $I_4(4,4;2)$. 
Note that an array of type $I_2(n,n;n)$ is equivalent to 
a pair of orthogonal Latin squares of order $n$.

%Let $T$ be a subset of $[k]$ where $T=\{i_1,i_2,\dots ,i_t\}$ and $i_1<i_2<\dots <i_t$. 
%Given ${\bf v}=(v_1,v_2,\dots ,v_k)\in [q]^k$, we define ${\bf v}|_T$ to be the vector $(v_{i_1},v_{i_2},\dots ,v_{i_t})$.     An $m\times n$ {\em row-column factorial design} $q^k$ of {\em strength} $t$ is an arrangement of the the elements of $\alpha \times q^k$ into an $m\times n$ array (which we call $I_k(m,n;q)$)   
%such that for each row (column) and for each subset $T$ of $Q$ of size $t$, 
%the set $\{{\bf v}_T: {\bf v}$ is in the row (respectively, column) $\}$  is a $[q]^t$ full-factorial design with replication $n/q^t$ (respectively, $m/q^t$).

We first review the impact of row-column factorial designs within experimental design literature.  
A {\em blocking factor} can be thought of as a partition of the blocks of a design, typically an equipartition (all subsets in the partition have equal size) with further properties of regularity to minimize estimation bias within the design structure. 
 Blocking factors for factorial designs have been well-studied (\cite{bailey1977patterns}, \cite{bailey1985factorial}, \cite{dean1980unified}, \cite{dean1992multidimensional}, \cite{kobilinsky1985confounding}). However, as mentioned in \cite{godolphin2019construction}, having two forms of blocking for a factorial design is less well-studied. 
 
Within experimental design a {\em row-column} design can refer to a variety of combinatorial designs, all with the property of being arranged in a rectangular array, where the rows and columns are typically (but not always) blocking factors. This is sometimes referred to as {\em double confounding} \cite{godolphin2019construction}.  To ensure that certain effects can be estimated without bias being introduced by the design, regularity conditions are imposed. For example, in a Latin square each symbol occurs once per row and once per column. 
While some examples of row-column factorial designs are given in \cite{yates1978design}, the earliest known example in the literature of a regular row-column factorial design is given in  \cite{rao1946confounded}, where an array of type $I_5(8,8;2)$ is featured. 
 In practice non-regular row-column factorial designs are also sometimes of use. In \cite{williams1996row}, a non-regular row-column factorial design is given which was used by the CSIRO Division of Forestry for a glasshouse experiment. Here the physical distance to the edge of the glasshouse is an important factor.

A {\em quasi-Latin square} is an $n\times n$ array such that for some $k>n$ which divides $n^2$, each entry from $[k]$ occurs $n^2/k$ times in the array, with no entry occurring more than once per row or column.  
Some of the literature on quasi-Latin squares features row-column factorial designs \cite{brien2012quasi}. Here if we consider the vectors as the entries, a row-column factorial design can be thought of as a quasi-Latin square if no vector occurs more than once in a row or column (necessarily, $m,n<2^k$).  
John and Lewis \cite{john1983factorial} describe a technique to cyclically generate some row-column factorial designs.  Examples and methods to construct row-column factorial designs are also given in \cite{choi2008confounded}, \cite{dash2013row} and \cite{cheng2013templates}. Wang \cite{wang2017orthogonal} constructs  $I_k(2^M,2^N;2)$ whenever $k=M+N$. A variation of row-column factorial designs is considered by \cite{doi:10.1080/03610926.2015.1122062}: a {\em generalized confounded row-column design} can be thought of as a factorial design arranged into a rectangular array where each cell contains a constant number of vectors. 
 
 Row-column factorial designs with two levels (that is, $q=2$) are studied in \cite{godolphin2019construction}. As well as the result in Theorem \ref{godolph} below, designs are also constructed to estimate paired interactions without confounding by row and column blocking factors.  

\begin{theorem} (\cite{godolphin2019construction}) Let $1\leq M\leq N$. A regular $2^M\times 2^N$ row-column factorial design $2^k$ exists if and only if $k\leq M+N$ and $(k,M,N)\neq (2,1,1)$. 
\label{godolph}
\end{theorem}

Our main result is generalizing the above to arrays of arbitrary dimensions and number of levels (that is, arbitrary $m$, $n$,  $q$ and $k$). 

	\begin{theorem} \label{thm:mainresult}
	Let $m\leq n$. 
		There exists a regular $m\times n$ row-column factorial design $q^k$ if and only if $q$ divides $m$, $q$ divides $n$, $q^k$ divides $mn$ 
unless: 
		\begin{enumerate}
			\item[i.] $k=q=m=2 $ and $n \equiv 2 \pmod{4} $.
			\item[ii.] $k=2$ and $ q=m=n=6$.
		\end{enumerate}
	\end{theorem}

We also show that the above exceptions are genuine exceptions; i.e. the above theorem describes necessary and sufficient existence conditions.

%The next lemma follows from the definition of a factorial design. 

%\begin{lemma}
%If $D$ is a regular  $m\times n$ {\em row-column factorial design} $q^k$, then:
%\begin{itemize} 
%\item $q^k\vert mn$;  
%\item $q\vert m$ and $q\vert n$.  
%\end{itemize}
%\end{lemma}   

% Two equipartitions ${\mathcal P}_1$ and ${\mathcal P}_2$ of the same set are said to be {\em orthogonal} if $|P_1\cap P_2|\leq 1$ for each  $P_1\in {\mathcal P}_1$ and $P_2\in {\mathcal P}_2$. 
% If we consider the equipartitions to be acting as two blocking factors, this orthogonality avoids  confounding between the two blocking factors. Given two orthogonal equipartitions ${\mathcal P}_1$ and ${\mathcal P}_2$ of the blocks of a design, we can organize blocks into a rectangular array, where rows correspond to elements of ${\mathcal P}_1$ and columns correspond to elements of   
% ${\mathcal P}_2$. 
 
% Hence when two orthogonal equipartitions are incorporated into an experimental design, the design is often referred to as a {\em row-column} design in the literature, also referred to as {\em double confounding}. This is quite a broad term!!! Really any combinatorial design in a rectangular array!  

We next describe the connection between regular row-column factorial designs and frequency rectangles. 
Given two vectors ${\bf v}=(v_0,v_1,\dots ,v_{k-1})$ and ${\bf w}=(w_0,w_1,\dots ,w_{\ell-1})$, we define ${\bf v}\oplus {\bf w}$ to be the 
concatenation of ${\bf v}$ and ${\bf w}$, that is:   
$${\bf v}\oplus {\bf w}:=(v_0,v_1,\dots ,v_{k-1},w_0,w_1,\dots ,w_{\ell-1}).$$ 
Next, let $A=[a_{ij}]$ and $B=[b_{ij}]$ be matrices of the same dimensions, with each entry of $A$ is a vector of dimension $k$ and each entry of $B$ is a vector of dimension $\ell$. Then  we define $C=A\oplus B$ to be the matrix given by $C=[c_{ij}:=a_{ij}\oplus b_{ij}]$.  
 
Now, an array of type $I_k(m,n;q)$ can be written in the form $F_0\oplus F_1\oplus \dots \oplus F_{k-1}$, where each entry of each $F_i$, $i\in [k]$, has dimension $1$. 
Since regularity is assumed, each element of $[q]$ occurs precisely $n/q$ times per row and $m/q$ times per column, for each of the arrays $F_i$, $i\in [k]$. 
These arrays are thus {\em frequency rectangles}.  

Formally, a {\em frequency rectangle} (sometimes known as an {\em $F$-rectangle}) of type 
%$F(n;q)$ is an $n\times n$ array where symbol $i\in [q]$ occurs $$ times per row and $\lambda_i$ times per column. 
%$F(n;\lambda_1,\lambda_2, \dots ,\lambda_{q})$  is an $n\times n$ array where symbol $i\in [q]$ occurs $\lambda_i$ times per row and %$\lambda_i$ times per column. 
%In the case that $\lambda_1=\lambda_2=\dots =\lambda_q=n/q$, the notation $F(n;n/q)$ is used. 
%Although much less considered in the literature, the idea of a frequency square easily generalizes to a rectangle.
%To this end (and slightly abusing the above notation to reduce the number of parameters by $1$), we define a {\em frequency rectangle} %of type 
$FR(m,n;q)$ is an $m\times n$ array such that each element of $[q]$ occurs $m/q$ times per row and $n/q$ times per column. 
Thus, we may write any array of type $I_k(m,n;q)$ as $F_0\oplus F_1\oplus \dots \oplus F_{k-1}$, where for each $i\in [k]$, $F_i$ is a frequency rectangle of type $FR(m,n;q)$. 
 We note here that frequency rectangles  in the literature (most often {\em frequency squares} or {\em $F$-squares} when $m=n$) may have different row/column frequencies for distinct symbols. In this paper we restrict ourselves to the regular case.    

Two frequency rectangles of type $FR(m,n;q)$ are orthogonal if, when superimposed, each ordered pair from $[q]\times [q]$ occurs exactly $mn/q^2$ times in the array. 
A set of pairwise orthogonal frequency rectangles are called {\em mutually orthogonal rectangles}. 
These have mostly been studied in the case $m=n$, where such structures are called Mutually Orthogonal Frequency Squares or MOFS). 

	The existence problem for pairs of MOFS has been completely classified; the following theorem is a special case of  \cite[p.~67]{laywine1998discrete}. The exceptions are precisely the two orders for which pairs of MOLS which do not exist, as known by Euler. 
	
	\begin{theorem}
		There exists a pair of MOFS of type $ F(n,n; q)$ if and only if $(n,q)\not\in \{(2,1),(6,1)\}$. 
		\label{thm:pair.of.MOFS}
	\end{theorem}

Hedayat, Raghavarao, et al. \cite{hedayat1975further} showed that if a set of $k$ MOFS of type $FR(n,n;q)$ exists then $k\leq (n-1)^2/(q-1)$. When $k$ meets this upper bound such a set is called {\em complete}. Complete sets of MOFS exist when $q=2$ and if there exists a Hadamard matrix of order $n$ \cite{federer1977existence}; otherwise they are only known to exist when $n$ is a prime power \cite{laywine2001table, li2014some, mavron2000frequency, street1979generalized}. A complete sets of MOFS 
does not exist when $q=2$ and $n\equiv 2\pmod{4}$ \cite{cavenaghwanless2020}.

Note that while an array of type $I_k(m,n;q)$ yields a set of $k$ mutually orthogonal frequency rectangles, the converse is not always true for $k\geq 3$, as seen below. Here we see a set of three MOFS (overlapped) which is not a row-column factorial design. 

\begin{table}[H]
    \begin{center}
        \begin{tabular}{c c c c c c }
				000 & 111 & 000 & 101 & 011 & 110\\
				111 & 000 & 000 & 011 & 110 & 101\\ 
				000 & 000 & 111 & 110 & 101 & 011\\
				101 & 011 & 110 & 010 & 100 & 001\\
				011 & 110 & 101 & 100 & 001 & 010\\
				110 & 101 & 011 & 001 & 010 & 100\\ 
		\end{tabular}
		\caption*{Three MOFS of type $ F(6,6;3) $}
			\label{tbl:3notsofs}
	\end{center}
\end{table}
%[[[ sequences $ 000, 011, 101, 110 $ appear six times, while the sequences $ 001, 111, 010, 100 $ appear only three times.]]]
However, if $k=2$ an $I_k(m,n;q)$ is equivalent to a pair of mutually orthogonal frequency rectangles of type $F(m,n;q)$.

	\begin{theorem} \cite{federer1984pairwise}
	Let $q$ divide $m$ and $n$. 
	If $q\not\in\{2,6\}$ or at least one of $n/q$, $m/q$ is even, 
		 there exists a pair of mutually orthogonal frequency rectangles of type $ F(m,n; q)$ (equivalently, an array of type $I_2(m,n;q)$). 
		\label{thm:pair.of.MOFRS}
	\end{theorem}

In Section 4, we generalize the previous theorem to find necessary and sufficient conditions for the existence of an array of type $I_2(m,n;q)$. We prove the remaining cases of Theorem \ref{thm:mainresult} in Section 5, using the recursive constructions from Section 2 and the finite field constructions from Section 3.  

% We often use the terminology of frequency squares in the proofs in this paper. 

%	\begin{lemma}
%		Let $ a = \max\{i \in \mathbb{N} : q^i \arrowvert mn \} $ then the number of strongly orthogonal frequency rectangles of type $FR(m,n; \ \nicefrac{n}{q}, \nicefrac{m}{q})$ must satisfy $ k \leq a $.
%		\label{lem:upperbound.SOFRS}
%	\end{lemma}
	
%	\begin{definition}
%			We say that a set of $ k $ frequency rectangles of type $FR(m,n; \ \nicefrac{n}{q}, \nicefrac{m}{q})$ is \textit{complete} if $ k = a $, where $ a $ is defined in Lemma \ref{lem:upperbound.SOFRS}.
%	\end{definition}

	\section{Recursive constructions}
	
	In this section we discuss ways in which row-column factorial designs can be built recursively. 
We begin with a straightforward lemma. 

\begin{lemma}
If there exists arrays of type $I_k(m,n;q)$ and $I_k(m',n;q)$ there exists an array of type $I_k(m+m,n;q)$. 
If there exists arrays of type $I_k(m,n;q)$ and $I_k(m,n';q)$ there exists an array of type $I_k(m,n+n';q)$.
\label{lem:glueing}
\end{lemma}

Next we consider a type of Kronecker Product. 	Let $ A=[a_{ij}]$ and $ B=[b_{ij}]$ be two arrays of sizes $ m \times n $ and $ u \times v $, respectively. The \textit{Kronecker product}, $ A \otimes B$, of $ A$ and $B $ is an $ mu \times nv $ array defined by:
	
	\begin{equation*}
	A \otimes B = \begin{pmatrix}
	a_{11}B & \ldots & a_{1n}B  \\
	\vdots  & 		 & \vdots   \\
	a_{m1}B & \ldots & a_{mn}B \\
	\end{pmatrix}
	\end{equation*}
	
	where $ a_{lk} B$ is a $ u \times v $ array with the entry $ (i,j) $ given by $(a_{lk}, b_{ij}) $. 
	
	It is easy to see that if $ F $ and $ F' $ are two frequency rectangles of type $ FR(m,n;q)$ and $ FR(m',n';q') $ respectively, then their Kronecker product $ F \otimes F' $ is a frequency rectangle of type $FR(mm',nn'; qq')$,
 where the entries of $[q]\times [q']$ are mapped to $[qq']$ by some bijection $f$.  
 In this fashion, let $D=F_0\oplus F_1\oplus \dots \oplus F_{k-1}$
	and $D'=F_0'\oplus F_1'\oplus \dots \oplus F_{k-1}'$ be two row-column factorial designs of types $I_k(m,n;q)$ and $I_k(m',n';q')$, respectively, where $F_i$ and 
$F_i'$ are frequency rectangles for each $i\in [k]$. 
Then we define $D\boxtimes D'$ to be the array given by  $(F_0\otimes F_0')\oplus (F_1\otimes F_1')\oplus \dots \oplus (F_{k-1}\otimes F_{k-1}')$.
	
%	The following lemma is analagous to a construction for orthogonal arrays (Theorem III.7.20 from \cite{Colbourn:2006:HCD:1202540}, originally \cite{bushka}). 

	\begin{lemma}
		\sloppy If $D$ and $D'$ are arrays of type $I_k(m,n;q)$ and $I_k(m',n';q')$, respectively, then 
		$D\boxtimes D'$, as defined above, is an array of type $I_k(mm',nn';qq')$. 
	\end{lemma}
	
	\begin{proof}
	It suffices to show that the entries of the cells of $D\boxtimes D'$
	form a factorial design. 
Let $D=F_0\oplus F_1\oplus \dots \oplus F_{k-1}$
	and $D'=F_0'\oplus F_1'\oplus \dots \oplus F_{k-1}'$ as above. 
	
	Consider any 
	 $ (\alpha_0, \alpha_1, \dots \alpha_{k-1})\in [qq']$ in a cell of 
	 $D\boxtimes D'$. 
	 Then for each $i\in [k]$, $\alpha_i = f(a_i,b_i)$, for some $ a_i$ and $ b_i$ belong to the symbol sets of $ F_i $ and $ F_i' $ respectively.
Since $D$ is of type $I_k(m,n;q)$, there are precisely $mn/q^k$ cells containing $a_i$ in $F_i$ for each $i\in [k]$. 
Similarly, there are exactly $m'n'/(q')^k$ cells containing $b_i$ in $F_i'$ for each $i\in [k]$.  	
 From the definition of the Kronecker product,  the sequence $(\alpha_0,  \alpha_1, \dots ,\alpha_{k-1})$ appears exactly $mm'nn'/(qq')^k$ times in $D\boxtimes D'$.
	\end{proof}

\begin{example}
Consider the two arrays $D$ and $D'$ of type $ I_3(4,2;2) $ and $ I_3(3,9;3) $ respectively.

\begin{table}[H]
		\begin{minipage}[h]{0.4\textwidth}
			\centering
			\begin{tabular}{cccc}
				000&011&101&110 \\
				111&100&010&001 \\
			\end{tabular}
			\caption*{$D^T$}
		\end{minipage}
		\begin{minipage}[h]{0.58\textwidth}
			\centering
			\begin{tabular}{ccc ccc ccc}
				000&011&022& 101&112&120& 202&210&221 \\
				111&122&100& 212&220&201& 010&021&002 \\
				222&200&211& 020&001&012& 121&102&110 \\
			\end{tabular}
			\caption*{$D'$}
		\end{minipage}		
	\end{table}

Then by taking the product $D \boxtimes D'$ and using the bijection $(a,b) \mapsto 3a + b $ to transform the symbol set, we get an array of type $I_3(12,18;6)$:
\begin{table}[H]
    \begin{center}
    \renewcommand{\arraystretch}{1.2}
    \resizebox{15.0cm}{!}{%
        \begin{tabular}{ccc ccc ccc ccc ccc ccc}
    000	&011&022&	101&112&120&	202&210&221&	333&344&355&	434&445&453&	535&543&554 \\
    111	&122&100&	212&220&201&	010&021&002&	444&455&433&	545&553&534&	343&354&335 \\
    222	&200&211&	020&001&012&	121&102&110&	555&533&544&	353&334&345&	454&435&443 \\
    033	&044&055&	134&145&153&	235&243&254&	300&311&322&	401&412&420&	502&510&521 \\
    144	&155&133&	245&253&234&	043&054&035&	411&422&400&	512&520&501&	310&321&302 \\
    255	&233&244&	053&034&045&	154&135&143&	522&500&511&	320&301&312&	421&402&410 \\
    303	&314&325&	404&415&423&	505&513&524&	030&041&052&	131&142&150&	232&240&251 \\
    414	&425&403&	515&523&504&	313&324&305&	141&152&130&	242&250&231&	040&051&032 \\ 
    525	&503&514&	323&304&315&	424&405&413&	252&230&241&    050&031&042&	151&132&140 \\
    330	&341&352&	431&442&450&	532&540&551&	003&014&025&	104&115&123&	205&213&224 \\
    441	&452&430&	542&550&531&	340&351&332&	114&125&103&	215&223&204&	013&024&005 \\
    552	&530&541&	350&331&342&	451&432&440&	225&203&214&	023&004&015&	124&105&113 \\
		\end{tabular}}
		\caption*{An array of type $ I_3(12,18;6) $}
		\end{center}
\end{table}
\end{example}	

\begin{corollary}
		\sloppy If there exist $ r $ arrays of types $ I_{k}(m_i, n_i; q_i) $, where $ i\in[r]$, then there exist an array of type $I_k(\prod_{i=0}^{r-1}m_i, \prod_{i=0}^{r-1}n_i  ;\prod_{i=0}^{r-1}q_i) $.
%, where $ k = \min\{k_1,\dots,k_m\} $. 
		\label{cor:kroneckerproduct.of.m.sofs}
	\end{corollary}

 Trivially there exists an array of type $I_k(m,n;1)$ for any integers $k,m,n$. The following corollary is then immediate. 
	
	\begin{corollary}
	    \sloppy 
	    If there exists an array of type $I_k(m,n;q)$, 
	    then there exists an array of type $I_k(mm',nn';q)$ for any integers $m',n'\geq 1$. 
	    \label{cor:blowup}
	\end{corollary}

	\section{Prime power constructions}
	
In this section, we construct row-column factorial designs via finite fields of prime power order $q$. It is implicitly understood that field elements are relabelled to elements of $[q]$ as a final step in construction.

	\begin{lemma}
		Let $M, N\geq 1$ and $q\geq 2$, with $(M,N,q)\neq (1,1,2)$. Then there exists a linearly independent set of $M +N $ polynomials:
		\begin{equation*}
		f_r(x_0, ... , x_{M+N-1}) = a_{r,0}x_0+ a_{r,1}x_1+ \cdots + a_{r,M+N-1}x_{M+N-1};\ r\in [M+N] 
		\end{equation*}
		over the field $ \mathbb{F}_q $ which satisfy the following two conditions for each $r\in [M+N]$:
		\begin{enumerate}
			\item[(i)]  $ (a_{r,0}, ... , a_{r,M-1}) \neq (0, ... , 0) $;  
			\item[(ii)]  $ (a_{r,M}, ... , a_{r,M+N-1} ) \neq (0, ... , 0) $.
		\end{enumerate}
		\label{lem:FR.Polynomilas.existence.}
	\end{lemma}	

\begin{proof}
	We split the proof into cases.
	
	\textbf{Case I:} When $ M=N=1 $ and $ q > 2 $.

	In this case we can take the following two polynomials:
	\begin{equation*}
	\begin{aligned}
	&f_0(x_0,x_1) & = \; & x_0 +  x_1& \\
	&f_1(x_0,x_1) & = \; & x_0 + \alpha x_1,& \\
	\end{aligned}
	\end{equation*}
	where $ \alpha $ is a non-zero element in $ \mathbb{F}_q $ other than identity.
	
	\textbf{Case II (a):} When $ N\geq 2 $ and $q$ is a power of $2$.
	
	Consider the identity matrix $ I_{M+N} $ of order $M+N$. By performing the following two row operations sequentially: 
	\begin{equation} \label{eq:row.operation}
	\begin{aligned}
	&R_0 + R_s \rightarrow R_s \hspace{0.5cm} \text{for each} \hspace{0.2cm} s \in \{1,2, \dots , M+N-1\}; \\
	&R_s + (R_{M+N-1}+R_{M+N}) \rightarrow R_s \hspace{0.5cm} \text{for each} \hspace{0.2cm} s \in [M], \\
	\end{aligned}
	\end{equation}
	we get the following matrix: 

    	\begin{equation*}
	\renewcommand{\arraystretch}{1.6}
	\left(
	\begin{array}{c@{}cccc:ccccc}
	\LF{\mathsmaller{R_0}}          & \UP{c_0}{1}  & \UP{c_1}{0} & \UP{\dots}{\dots \ } & \ \ \UP{c_{M-1}}{0 \ \ } \ \ \ & \ \ \ \UP{c_{M}}{0} \ \ \  & \ \UP{c_{M+1}}{0 \ } \  & \ \ \UP{\dots}{\dots \ } \ \ & \ \ \UP{c_{M+N-2}}{1 } \ \ \ & \ \ \ \ \ \UP{c_{M+N-1}}{1 \ \ } \ \ \ \\ 
	\LF{\mathsmaller{R_1}}              & 1  		& 1 	 &{\dots} & 0 \ \ \ 	 & \ \ \ 0 \ \ \  	   & 0 		& \dots  & 1 \   &  \ \ 1 \ \  \\ 
	\LF{\vdots} 	      & \vdots  & \vdots & \ddots & \vdots \ \ \ & \ \ \ \vdots \ \ \  & \vdots & \ddots & \ \vdots \ \    & \ \ \vdots \ \ \\
	\LF{\mathsmaller{R_{M-1}}}          	  & 1  		& 0 	 &{\dots} & 1 \ \ \ 	 & \ \ \ 0 \ \ \  	   & 0 		& \dots  & 1 \ \   & \ \ 1 \ \ \\ \hdashline
	\LF{\mathsmaller{R_{M}}}          & 1  		& 0 	 &{\dots} & 0 \ \ \ 	 & \ \ \ 1 \ \ \  	   & 0 		& \dots  & 0  \ \  & \ \ 0 \ \ \\
	\LF{\mathsmaller{R_{M+1}}}          & 1  		& 0 	 &{\dots} & 0 \ \ \ 	 & \ \ \ 0 \ \ \  	   & 1 		& \dots  & 0 \ \   & \ \ 0 \ \ \\
	\LF{\vdots}           & \vdots  & \vdots & \ddots & \vdots \ \ \ & \ \ \ \vdots \ \ \  & \vdots & \ddots & \vdots \ \    & \ \ \vdots \ \ \\ 
	\LF{\mathsmaller{R_{M+N-1}}}          & 1  		& 0 	 &{\dots} & 0 \ \ \ 	 & \ \ \ 0 \ \ \  	   & 0 		& \dots  & 0 \ \   & \ \ 1 \ \ \\
	\end{array}\right)
	\end{equation*}
	
	Now corresponding to each row $ R_s = (r_{s,\mathsmaller{0}},\dots, r_{s,(M+N-1)})$ of the above matrix we define a polynomial $ f_s = r_{s,\mathsmaller{0}}x_0 + \dots + r_{s,(M+N-1)x_{M+N-1}}  $ in $ \mathbb{F}_q $, where, $ s \in [M+N]$. Then these polynomials satisfy the conditions (i) and  (ii)  and are linearly independent. 
	
	\textbf{Case II (b):} When $ N \geq 2 $ and $q$ is not a power of $2$.
	
	In this case again take the identity matrix $ I_{M+N} $ and by replacing the second row operation  in (\ref{eq:row.operation}) by $ R_s + R_{M+N} \rightarrow  R_s $ for each $ s\in [M]$, we get the following matrix: 
	
	\begin{equation*} 
	\renewcommand{\arraystretch}{1.6}
	\left(
	\begin{array}{c@{}cccc:cccc}
	\LF{\mathsmaller{R_0}}          & \UP{c_0}{2}  & \UP{c_1}{0} & \UP{\dots}{\dots} & \ \ \UP{c_{M-1}}{0 \ \ } \ \ \ & \ \ \ \UP{c_{M}}{0} \ \ \  & \ \UP{c_{M+1}}{0 \ } \  & \ \ \UP{\dots}{\dots \ } \ \ &\ \ \ \UP{c_{M+N-1}}{1 \ \ } \ \ \ \\  
	\LF{\mathsmaller{R_1}}              & 2  		& 1 	 &{\dots} & 0 \ \ \ 	 & \ \ \ 0 \ \ \  	   & 0 		& \dots  & 1 \ \ \\ 
	\LF{\vdots} 	      & \vdots  & \vdots & \ddots & \vdots \ \ \ & \ \ \ \vdots \ \ \  & \vdots & \ddots & \vdots \ \ \\
	\LF{\mathsmaller{R_{M-1}}}          	  & 2  		& 0 	 &{\dots} & 1 \ \ \ 	 & \ \ \ 0 \ \ \  	   & 0 		& \dots  & 1 \ \ \\ \hdashline
	\LF{\mathsmaller{R_{M}}}          & 1  		& 0 	 &{\dots} & 0 \ \ \ 	 & \ \ \ 1 \ \ \  	   & 0 		& \dots  & 0 \ \ \\
	\LF{\mathsmaller{R_{M+1}}}          & 1  		& 0 	 &{\dots} & 0 \ \ \ 	 & \ \ \ 0 \ \ \  	   & 1 		& \dots  & 0 \ \ \\
	\LF{\vdots}           & \vdots  & \vdots & \ddots & \vdots \ \ \ & \ \ \ \vdots \ \ \  & \vdots & \ddots & \vdots \ \ \\
	\LF{\mathsmaller{R_{M+N-1}}}          & 1  		& 0 	 &{\dots} & 0 \ \ \ 	 & \ \ \ 0 \ \ \  	   & 0 		& \dots  & 1 \ \ \\
	\end{array}\right)
	\end{equation*}
	
	It is easy to see that the corresponding polynomials are linearly independent in $ \mathbb{F}_q $ and satisfy the conditions (i) and (ii).
	\end{proof}
	
	\begin{theorem}
Let $q\geq 2$ be a prime power. 
Let $M, N\geq 1$ and $(M,N,q)\neq (1,1,2)$. 
		There exists an array  of type $I_{M+N}(q^M, q^N; q) $. 
		\label{thm:SOFRS.primepowers}
	\end{theorem}
	
	\begin{proof}
		\sloppy First we describe a method to construct a frequency rectangle of type $ FR(q^M, q^N; q) $ corresponding to polynomial $f_r$ for each $r\in [M+N]$, as given  by  Lemma \ref{lem:FR.Polynomilas.existence.}.
		
		Label the rows and columns of a $ q^{M} \times q^{N} $ array, respectively, by using the set of all $M$-tuples  and $N$-tuples over the field  $ \mathbb{F}_q $. Now consider a polynomial
		\begin{equation*}
		f_r(x_0, ... , x_{M+N-1}) = a_{r,0}x_{r,0}+ \cdots + a_{r,M+N-1}x_{r,M+N-1}
		\end{equation*}
		over the field $ \mathbb{F}_q $ that satisfies the conditions given in Lemma \ref{lem:FR.Polynomilas.existence.}. 
		We place the element $ f(b_0,...,b_{M-1},c_0, ..., c_{N-1}) $ in the intersection of row $ (b_0,..., b_{M-1}) $ and column $ (c_0, ... , c_{N-1}) $ of the $ q^{M} \times q^{N} $ array.
		
		Now we show that the array obtained in this way is a frequency rectangle of type $ FR(q^M, q^N; q) $, that is every element of $ \mathbb{F}_q $ appears exactly $ q^{N-1} $ times in each row and $ q^{M-1} $ times in each column. Consider a row which is labelled by $ (b_0, ... , b_{M-1}) $ and take an element $ \alpha \in  \mathbb{F}_q $. For this row, the equation 
		\begin{gather*}
		f_{r}(b_0, ... , b_{M-1},x_{M},..., x_{M+N-1}) = \alpha
		\end{gather*}
		reduces to the equation
		\begin{gather}
		K + a_{r,M}x_{M} + ... + a_{r,M+N-1}x_{M+N-1} = \alpha
		\label{eq:FR.proof}
		\end{gather}
		where $ K $ is a constant. Now by axiom (ii) of Lemma \ref{lem:FR.Polynomilas.existence.} there exists $i\in \{M,M+1,\dots ,M+N-1\}$ such that $ a_{r,i} \neq 0 $. We solve the equation (\ref{eq:FR.proof}) for $ x_i $:
		\begin{gather}
		x_i = \frac{1}{a_{r,i}} (\alpha - K - a_{r,M}x_{M} - \cdots - a_{r,i-1}x_{i-1}-a_{r,i+1}x_{i+1} - \cdots - a_{r,M+N-1}x_{M+N-1}).
		\label{eq:FRproof2}
		\end{gather}
		Since there are $ q $ elements in $ \mathbb{F}_q $ and the $ N-1 $ variables on the right side of (\ref{eq:FRproof2}) can take any value from $ \mathbb{F}_q $, the equation (\ref{eq:FRproof2}) has exactly $ q^{N-1} $ solutions in $ \mathbb{F}_q $. This implies that the symbol $ \alpha $ appears exactly at $ q^{N-1} $ places in the row $ (b_0, ... , b_{M-1}) $. By a similar argument we can prove that each symbol appears exactly $ q^{M-1} $ times in each column. Thus the resulting array is a frequency rectangle of type $ FR(q^M, q^N; q) $.
		
		Now to construct an array of type  $I_{M+N}(q^M, q^N; q) $. Consider a set of $M+N$ linearly independent polynomials
		\begin{equation*}
		f_r(x_0, ... , x_{M+N-1}) = a_{r,0}x_0+ \cdots + a_{r,M+N-1}x_{M+N-1}; \hspace{0.5cm} r\in [M+N] 
		\end{equation*}
		over the field $ \mathbb{F}_q $, such that the coefficients satisfy the conditions (i) and (ii) of Lemma \ref{lem:FR.Polynomilas.existence.}. 
		As above, for each 
		 $r\in [M+N]$ we obtain a frequency rectangle $F_r$ of type $FR(q^M,q^N;q)$ 
		
		It remains to show that $F_0\oplus F_1\oplus \dots \oplus F_{M+N-1}$ is an array of type 
		$I_{M+N}(q^M, q^N; q) $.
		To this end, consider any $(\alpha_0, \alpha_1,...,\alpha_{M+N-1})\in (\mathbb{F}_q)^{M+N} $. Since the polynomials above are linearly independent, the system of equations:
		\begin{equation*}
		\begin{aligned}
		&a_{0,0}x_0 & + \ \cdots \ + &\ a_{0,M+N-1}x_{M+N-1} \ \ \ \ &= \ \ \alpha_0 \\
		&a_{1,0}x_0 & + \ \cdots \ + &\ a_{1,M+N-1}x_{M+N-1} \ \ \ \ &= \ \ \alpha_1 \\
		& \ \ \ 	\vdots & \ddots \ \ \ \   &\  \ \ \ \ \ \  \vdots  &  \ \  \vdots  \ \\
		&a_{M+N-1,0}x_0 & + \ \cdots \ +& \ a_{M+N-1,M+N-1}x_{M+N-1}  &= \ \ \alpha_{M+N-1}\\
		\end{aligned}
		\end{equation*}
		
		has rank $M+N$ and therefore has a unique solution in $ \mathbb{F}_q $, which shows that $(\alpha_0, \alpha_1,...,\alpha_{M+N-1})$ appears in exactly one cell of the array constructed. 
	\end{proof}
	
	Corollary \ref{cor:blowup} implies the following: 
	
	\begin{corollary}
		If $(q,M,N)\neq (2,1,1)$, there exist an array of type $I_{M+N}(q^Mb_1, q^Nb_2; q) $, for any prime power $q$.
\label{cor:blowitup}	
\end{corollary}

\subsection{``Sudoku'' Frequency Rectangles} \label{sec:sudokuFRS}

In this subsection we take the construction above and take it one step further. Specifically, we show in Theorem \ref{thm:SOFRS.i+j+1}
that if $q$ divides $b_1b_2$ and $q$ is a prime power, there exists an array of type $I_{M+N+1}(q^Mb_1, q^Nb_2; q)$. 
	
First we describe a Latin square which has a Sudoku-type property with 
$q=q_1q_2$ symbols, where $q_1$ and $q_2$ are positive integers and the symbol set is taken to be $[q]$. That is, such a Latin square can be partitioned into $q_1\times q_2$ subarrays containing each element of $[q]$.

\begin{theorem}
\sloppy 
Let $q_1,q_2\geq 1$. 
Then there exists a Latin square $L(q_1,q_2)$ of order $q_1q_2$ such that
for each  $i\in [q_1]$ and $j\in [q_2]$,  
 the set of cells 
$$\{(i',j')\mid i\equiv i' \pmod{q_1},j\equiv j' \pmod{q_2}\}$$
contain each entry from $[q_1q_2]$ exactly once.
%; and (b) the set of cells 
%$$\{(i',j')\mid 
%\lfloor i/q_1\rfloor = \lfloor i'/q_1\rfloor,   
  %      \lfloor j/q_2\rfloor = \lfloor j'/q_2\rfloor \} $$ 
%contain each entry from $[q_1q_2]$ exactly once. 
\end{theorem}

\begin{proof}
Let $q=q_1\times q_2$, 
Consider the following array of size $ q \times q $  partitioned into $ q $ subarrays $ S_0, S_1, ..., S_{q-1} $, each of size $ q_1 \times q_2 $: 
\begin{table}[H]
	\begin{center}
		\begin{tabular}{c|ccc|ccc|ccp{0.7cm}ccc|ccc|}
			\multicolumn{1}{c}{} &\multicolumn{1}{c}{$\leftarrow$}&	$ q_2 $	  &\multicolumn{1}{c}{$\rightarrow$}&$ \leftarrow $ & $ q_2 $	 	  &\multicolumn{1}{c}{$\rightarrow$}&	\multicolumn{6}{c}{\ldots}	  &\multicolumn{1}{c}{$\leftarrow$}&	$ q_2 $	  &\multicolumn{1}{c}{$\rightarrow$} \\ \cline{2-16}	
			$\uparrow$ &&		  &&&	 	  &&&		 &&&&&&			 & \\
			$ q_1 $  &&$ S_0 $ &&& $ S_1 $ && \multicolumn{6}{c|}{\ldots} & \multicolumn{3}{c|}{$ S_{q_1-1}$}  \\  
			$ \downarrow $	&&		  &&&	 	  &&&		 &&&&&&			 & \\ \cline{2-16}
			$\uparrow$ &&		  &&&	 	  &&&		 &&&&&&			 & \\
			$ q_1 $  & \multicolumn{3}{c|}{$S_{q_1}$} & \multicolumn{3}{c|}{$ S_{q_1 + 1}$} & \multicolumn{6}{c|}{\ldots} & \multicolumn{3}{c|}{$ S_{2q_1-1}$}  \\  
			$ \downarrow $	&&		  &&&	 	  &&&		 &&&&&&			 & \\ \cline{2-16}
			&&		  &&&	 	  &&&		 &&&&&&			 & \\
			\vdots && \vdots &&& \vdots && \multicolumn{6}{c|}{$\ddots$} && \vdots & \\  
			&&		  &&&	 	  &&&		 &&&&&&			 & \\ \cline{2-16}
			$\uparrow$ &&		  &&&	 	  &&&		 &&&&&&			 & \\
			$ q_1 $  &\multicolumn{3}{c|}{$ S_{(q-q_1)}$} & \multicolumn{3}{c|}{$ S_{(q-q_1)+1} $} & \multicolumn{6}{c|}{\ldots} & \multicolumn{3}{c|}{$ S_{q-1} $}  \\  
			$ \downarrow $	 &&		  &&&	 	  &&&		 &&&&&&			 & \\ \cline{2-16}
		\end{tabular}
		\caption*{$L(q_1,q_2)$}
	\end{center} 
\end{table}
We next describe the structure of $ S_0 $. Arrange the $ q $ symbols in the sub array $ S_0 $ in ascending order starting from column $ 1 $. Thus $ S_0 $ has the following form:

\begin{table}[H]
	\begin{center}
		\renewcommand{\arraystretch}{1.3}
		\begin{tabular}{|ccccc|}
			\hline
			0&$ q_1 $&\multicolumn{2}{c}{\ldots}&$ (q-q_1) $ \\
			1&$ q_1 +1 $&\multicolumn{2}{c}{\ldots}&$ (q-q_1)+1 $ \\
			&&&&\\
			\vdots&\vdots&\multicolumn{2}{c}{$\ddots$}&\vdots \\
			&&&&\\
			$ q_1-1 $&$ 2q_1-1 $&\multicolumn{2}{c}{\ldots}&$ q-1 $ \\\hline
		\end{tabular}
		\caption*{$ S_0 $}
	\end{center} 
\end{table}

Now we define the subarray $ S_i $ to be the array obtained by adding the symbol $ i \pmod{q}  $ to each cell of $ S_0 $. 
Finally arrange these $ S_i $'s as described above to obtain $L(q_1,q_2)$. 
\end{proof}

\begin{example}
	We exhibit the previous construction in the case $ q=6, q_1 = 2 $ and $ q_2 = 3$: 

\begin{table}[H]
	\begin{center}
	\setlength{\tabcolsep}{10pt}
		\renewcommand{\arraystretch}{1.2}
		\begin{tabular}{|ccc|ccc|} \hline
			0&2&4 & 1&3&5 \\
			1&3&5 & 2&4&0 \\\hline
			2&4&0 & 3&5&1 \\
			3&5&1 & 4&0&2 \\\hline
			4&0&2 & 5&1&3 \\
			5&1&3 & 0&2&4 \\\hline
		\end{tabular}
	\caption*{$L(2,3)$}
	\end{center} 
\end{table}

\end{example}

Now, a Latin square of order $q$ is also an array of type $I_1(q,q;q)$. Thus, from Corollary \ref{cor:blowup}, we have the following corollary.  
\begin{corollary}
For any integers $\mu,\lambda,q_1,q_2\geq 1$, 
there exists a frequency rectangle of type $FR(q \mu, q \lambda; q)$ (where $q=q_1q_2$) 
such that
for each  $i\in [\mu q_1]$ and $j\in [\lambda q_2]$,   the set of cells 
$$\{(i',j')\mid i\equiv i' \pmod{\mu q_1},j\equiv j' \pmod{\lambda q_2}\}$$
contain each entry from $[q_1q_2]$ exactly once.
%; and (b) the set of cells 
%$$\{(i',j')\mid 
%\lfloor i/(\mu q_1)\rfloor = \lfloor i'/(\mu q_1)\rfloor,   
 %       \lfloor j/(\lambda q_2)\rfloor = \lfloor j'/(\lambda q_2)\rfloor \} $$ 
%contain each entry from $[q_1q_2]$ exactly once. 
\label{cor:sudokurec}
\end{corollary}

\begin{example}
If $L$ is the Latin square $L(2,3)$, then $L\boxtimes I_1(2,3;1)$ yields a frequency rectangle of type $FR(12,18;6)$. The entries in bold show the elements of $[6]$ occurring in cells 
of the form $(i,j)$ where $i\equiv 1\pmod{4}$ and $j\equiv 2\pmod{9}$. 

\begin{table}[H]
	\begin{center}
	\setlength{\tabcolsep}{7pt}
		\renewcommand{\arraystretch}{1.2}
		\begin{tabular}{|ccccccccc|ccccccccc|} \hline
			0&0&0&2&2&2&4&4&4 & 1&1&1&3&3&3&5&5&5 \\
			0&0&{\bf 0}&2&2&2&4&4&4 & 1&1&{\bf 1}&3&3&3&5&5&5 \\
			1&1&1&3&3&3&5&5&5 & 2&2&2&4&4&4&0&0&0 \\
			1&1&1&3&3&3&5&5&5 & 2&2&2&4&4&4&0&0&0 \\\hline
			2&2&2&4&4&4&0&0&0 & 3&3&3&5&5&5&1&1&1 \\
			2&2&{\bf 2}&4&4&4&0&0&0 & 3&3&{\bf 3}&5&5&5&1&1&1 \\
			3&3&3&5&5&5&1&1&1 & 4&4&4&0&0&0&2&2&2 \\
			3&3&3&5&5&5&1&1&1 & 4&4&4&0&0&0&2&2&2 \\\hline
			4&4&4&0&0&0&2&2&2 & 5&5&5&1&1&1&3&3&3 \\
			4&4&{\bf 4}&0&0&0&2&2&2 & 5&5&{\bf 5}&1&1&1&3&3&3 \\
			5&5&5&1&1&1&3&3&3 & 0&0&0&2&2&2&4&4&4 \\
			5&5&5&1&1&1&3&3&3 & 0&0&0&2&2&2&4&4&4 \\\hline
		\end{tabular}
	\caption*{A frequency rectangle of type $FR(12,18;6)$ by Corollary \ref{cor:sudokurec}}
	\end{center} 
\end{table}

\end{example}

Before we prove Theorem 
	\ref{thm:SOFRS.i+j+1}, we require the following number-theoretic observation. 

	\begin{lemma}
		Let $ b_1, b_2 $ and $ q $ be positive integers such that $ q $ divides the product $ b_1b_2 $. Then there exist positive integers $ q_1 $ and $ q_2 $ such that $ q_1 q_2 = q$ and $ q_1 $ divides $ b_2 $ and $ q_2 $ divides $ b_1 $.
		\label{lem:FRq1q2}
	\end{lemma}
	
	\begin{proof}
		Let $ q = p^{s_0}_0p^{s_1}_1 \ldots p^{s_{m-1}}_{m-1} $ be the prime factorization of $ q $. Since $ q $ divides $ b_1b_2 $, $ b_1 $ and $ b_2 $ must be of the form:
		\begin{equation*}
		\begin{aligned}
		b_1 &= B_1 \:  p^{\alpha_0}_0 p^{\alpha_1}_1 \ldots p^{\alpha_{m-1}}_{m-1},\\
		b_2 &= B_2 \:  p^{\beta_0}_0 p^{\beta_1}_1 \ldots p^{\beta_{m-1}}_{m-1},\\
		\end{aligned}
		\end{equation*}
		
		where $ p_i $ does not divide $ B_j $ and $ \alpha_i + \beta_i \geq s_i $ for all $ i\in[m]$ and $ j\in\{1,2\}$. Let
		
		\begin{equation*}
		\begin{aligned}
		q_1 &=  p^{u_0}_0 p^{u_1}_1 \ldots p^{u_{m-1}}_{m-1} \\
		\end{aligned}
		\end{equation*}
		and
		\begin{equation*}
		\begin{aligned}
		q_2 &=  p^{t_0}_0 p^{t_1}_1 \ldots p^{t_{m-1}}_{m-1}, \\
		\end{aligned}
		\end{equation*}
		
where $t_i:=\max \{0, s_i - \beta_i \}$ and $ u_i = s_i - t_i $ for all $ i \in[m]$.

Since $\alpha_i + \beta_i \geq s_i $, $t_i\leq \alpha_i$ for each $i\in[m]$, which implies that $q_2 $ divides $ b_1 $.
Also $ s_i - \beta_i \leq t_i  $ implies that $ u_i = s_i - t_i \leq \beta_i $ and thus $ q_1 $ divides $ b_2 $.
Finally observe that $q_1 q_2 = q$.  
	\end{proof}

	\begin{theorem}
		Let $q$ be a divisor of $b_1b_2$. If there exist an array of type  $I_{M+N}(q^M, q^N;q)$, then there exists an array of type $I_{M+N+1}(q^Mb_1, q^Nb_2; q) $. 
		\label{thm:SOFRS.i+j+1}
	\end{theorem}
	
	\begin{proof}
		By Lemma \ref{lem:FRq1q2} we can choose $ q_1,  q_2  $ such that $ q_1 q_2 = q $ and $ q_1 $ divides $ b_2 $ and $ q_2 $ divides $ b_1 $. 
%Take the Sudoku Latin square of size $ q $ as described in section \ref{sec:sudokuFRS} with parameters $ q_1  $ and $ q_2 $ and blow up each cell to an array of size $ \nicefrac{q^i}{q_1} \times \nicefrac{q^j}{q_2} $. 

Let $I'$ be the array $I_1(q_2,q_1;1)\boxtimes I_{M+N}(q^M,q^N;q)$ as shown in Table 	\ref{tbl:sudoku}.

		\begin{table}[H]
			\begin{center}
				\resizebox{13.0cm}{!}{%
					\begin{tabular}{c|ccc|ccc|ccp{0.7cm}ccc|ccc|}
						\multicolumn{1}{c}{} &\multicolumn{1}{c}{$\leftarrow$}&	$ q^N $	  &\multicolumn{1}{c}{$\rightarrow$}&$ \leftarrow $ & $ q^N $	 	  &\multicolumn{1}{c}{$\rightarrow$}&	\multicolumn{6}{c}{$\ldots$}	 &\multicolumn{1}{c}{$\leftarrow$}&	$ q^N $	  &\multicolumn{1}{c}{$\rightarrow$} \\ \cline{2-16}	
						$\uparrow$ &&		  &&&	 	  &&&		 &&&&&&			 & \\
						$ q^M$  &&$ I_{M+N} $ &&& $ I_{M+N} $ && \multicolumn{6}{c|}{\ldots} && $ I_{M+N} $  &  \\  
						$ \downarrow $	&&		  &&&	 	  &&&		 &&&&&&			 & \\ \cline{2-16}
						$\uparrow$ &&		  &&&	 	  &&&		 &&&&&&			 & \\
						$ q^M$  & \multicolumn{3}{c|}{$ I_{M+N} $ } & \multicolumn{3}{c|}{$ I_{M+N} $} & \multicolumn{6}{c|}{\ldots} & \multicolumn{3}{c|}{$ I_{M+N} $ }  \\  
						$ \downarrow $	&&		  &&&	 	  &&&		 &&&&&&			 & \\ \cline{2-16}
						&&		  &&&	 	  &&&		 &&&&&&			 & \\
						\vdots && \vdots &&& \vdots && \multicolumn{6}{c|}{$\ddots$} && \vdots & \\  
						&&		  &&&	 	  &&&		 &&&&&&			 & \\ \cline{2-16}
						$\uparrow$ &&		  &&&	 	  &&&		 &&&&&&			 & \\
						$ q^M $  &\multicolumn{3}{c|}{$ I_{M+N} $} & \multicolumn{3}{c|}{$ I_{M+N} $} & \multicolumn{6}{c|}{\ldots} & \multicolumn{3}{c|}{$ I_{M+N} $}  \\  
						$ \downarrow $	 &&		  &&&	 	  &&&		 &&&&&&			 & \\ \cline{2-16}
				\end{tabular}}
				\caption{$I'= I_1(q_2,q_1;1)\boxtimes I_{M+N}(q^M,q^N;q)$}
				\label{tbl:sudoku}
			\end{center} 
			\end{table}

Applying Corollary \ref{cor:sudokurec}
with $\mu=q^{M-1}q_2$ and $\lambda=q^{N-1}q_1$, 
 there exists  a rectangular array $J'$ of type $I_1(q^Mq_2, q^Nq_1;q)$ such that   
for each  $i\in [q^M]$ and $j\in [q^N]$, the set of cells 
$$\{(i',j')\mid i\equiv i' \pmod{q^M},j\equiv j' \pmod{q^N}\}$$
contain each entry from $[q_1q_2]$ exactly once.
It follows that the array $ I'\oplus J' $ contains each sequence of length $ M+N+1 $ exactly once. 
Thus $I'\oplus J'$ is of type $I_{M+N+1}(q^Mq_2,q^Nq_1;q)$. 
Finally, by Corollary \ref{cor:blowup}, $I_1(b_1/q_2,b_2/q_1;1) \boxtimes (I' \oplus J')$ is an array of type
 $ I_{M+N+1}(q^Mb_1, q^Nb_2; q) $. 	
	\end{proof}

\section{The case $k=2$}

In this section we prove Theorem \ref{thm:mainresult} in the case $k=2$. 
%(Theorem \ref{thm:SOFRS.pairs.existence} below).
 Given Theorem \ref{thm:pair.of.MOFRS}, it suffices to consider the  existence of an array of type $I_2(m,n;q)$ only in the case $q\in \{2,6\}$ and $m/q$ and $n/q$ odd. 
%We have two special cases here that we want to discuss separately. One %is when $ q =2 $ and the other one is $ q=6 $. 
From Theorem \ref{thm:pair.of.MOFS}, arrays of type $I_2(2,2;2)$ and 
 $I_2(6,6;6)$ do not exist. We next give another 
  non-existence result. 
	
	\begin{lemma} \label{lem:non.existence.of.pairs.q=2}
		There does not exist an array of type $I_2(2, n; 2)$ whenever $n/2$ is odd. 
	\end{lemma}
	
	\begin{proof}
		Consider a frequency rectangle $ F $ of type $ FR(2, n;2) $. By permuting columns we can assume $ F $ is in the following form: 
		\begin{table}[H]
			\begin{center}
				\renewcommand{\arraystretch}{1.2}
				\begin{tabular}{|cccc|cccc|}
					\hline
					0&0& \ldots &0&1&1& \ldots &1 \\\hline
					1&1& \ldots &1&0&0& \ldots &0 \\\hline
				\end{tabular}
			\end{center} 
		\end{table}
	
	Now consider any other frequency square $ F' $ of type $ FR(2, n; 2) $. Now since $ n \equiv 2 \pmod{4} $, $ F' $ contains at least $ \lfloor \nicefrac{n}{4}  \rfloor + 1 $ symbols of the same type (say $ 0 $) in the first $ n/2 $ cells of its first row. This implies there are at least $ \lfloor \nicefrac{n}{4}  \rfloor + 1 $ $ 1's $ in the second half of the first row and consequently we have $ \lfloor \nicefrac{n}{4}  \rfloor + 1 $ $0's $ in the second half of its second row. Thus if we superimpose $ F $ and $ F' $, we get at least $ 2 \times \lfloor \nicefrac{n}{4}  \rfloor + 2 > n/2 = 2n/4 $ ordered pairs of type $ (0,0) $. Which shows $ F $ and $ F' $ are not orthogonal.
	\end{proof}

	\begin{lemma}
	Let $m/2$ and $n/2$ be odd where $n\geq m>2$. 
		Then there exists an array of type $I_2(m, n; 2)$.
\label{lem:twoo}	
\end{lemma}
	
	\begin{proof}
		Let $ m = 2l_1 $ and $ n = 2l_2 $, where $ l_1 $ and $ l_2 $ are odd and $l_2\geq l_1>1$. Let $ l_2 = l_1 + 2t $. Now by Theorem \ref{thm:pair.of.MOFS} there exists 
an array of type $I_2(2l_1,2l_1; 2)$. 
By Theorem \ref{thm:SOFRS.primepowers} and Corollary \ref{cor:blowup} there exists an array of type 
$I_2(2l_1,4t;2)$. 
Thus by Lemma \ref{lem:glueing}, there exists 
an array of type $I_2(2l_1,2l_2;2)$. 
	\end{proof}

	Next we consider when $ q=6 $. 
	%We know that (see \cite{keedwell2015latin}) there does not exist %an array $I_2(6,6; 6) $, equivalently a pair of MOLS of order %$6$. 
	An array of type $I_2(6,12;6)$ exists by Theorem \ref{thm:pair.of.MOFS}. 
	We also exhibit an array of type $I_2(6,18;6)$: 

\begin{table}[H]
	\begin{center}
		\renewcommand{\arraystretch}{1.2}
		\resizebox{14.0cm}{!}{%
		\begin{tabular}{cccccccccccccccccc}
			13&24&35&40&51&02&15&24&30&43&51&02&10&24&33&45&51&02 \\
			34&43&01&52&20&15&34&45&01&52&23&10&34&40&01&52&25&13 \\
			43&32&10&25&04&53&41&32&13&20&04&55&41&32&15&23&04&50 \\
			22&11&54&03&45&30&22&11&54&05&40&33&22&11&54&00&43&35 \\
			50&05&23&31&12&44&53&00&25&31&12&44&55&03&20&31&12&44 \\
			04&50&42&14&33&21&00&53&42&14&35&21&03&55&42&14&30&21 \\
		\end{tabular}}
		\caption*{An array of type $I_2(6,18; 6)$.}
	\end{center} 
\end{table}

By Lemma \ref{lem:glueing}, we thus obtain the following. 
	\begin{lemma}
		There exists an array of type $I_2(6l_1, 6l_2; 6) $ if and only if $(l_1,l_2)\neq (1,1)$.
\label{lem:sixx}	
\end{lemma}

\section{The case $k\geq 3$.}

%	By Theorem \ref{thm:SOFRS.pairs.existence}, 
	It now suffices to prove the case $k\geq 3$ in order to prove Theorem \ref{thm:mainresult}.

	\begin{theorem} \label{thm:complete.SOFRS.mod4}
Let $k\geq 3$, $q\arrowvert m$, $q\arrowvert n$ and $q^k \arrowvert mn$.  Then there exist an array of type $I_k(m,n; q)$.
	%	Suppose $ m=q\lambda_2 $ and $ n=q\lambda_1 $, where $ n $ and $ q $ are not simultaneously congruent to 2 modulo 4 (either $ n \not \equiv q \pmod{4} $ or $ q \not \equiv 2 \pmod{4} $). 
	\end{theorem}

	\begin{proof}
Trivially, if an array of type $I_k(m,n;q)$ exists, then an array of type $I_{\ell}(m,n;q)$ exists for each $1\leq \ell<k$. Thus we may assume that 
$k = \max \{t : q^t \arrowvert mn \} $.
		Let $ mn = q^kb$. 
%Thus we require a set of $ k $  SOFRS of type $ FR(m,n; \ \lambda_1, \lambda_2) $ to form a complete set (see Lemma \ref{lem:upperbound.SOFRS}).
		
Consider the prime factorization of $ q $: 
		\begin{equation*}
		q = p_0^{s_0}p_1^{s_1} \dots p_{l-1}^{s_{l-1}}. 
		\end{equation*}
				
		For each  $ r\in [l]$, let $ i_r = \max \{t : q_r^t \arrowvert m \}$ and $ j_r = \max \{t : q_r^t \arrowvert n \} $, where $ q_r = p_r^{s_r} $. Thus $ m $ and $ n $ can be expressed as: 
		\begin{equation} \label{eq:SOFRS.mn}
		m = q_0^{i_0} \dots q_{l-1}^{i_{l-1}} \ p_0^{\alpha_0} \dots p_{l-1}^{\alpha_{l-1}} \ a_1, \qquad   n = q_1^{j_1} \dots q_l^{j_l} \ p_0^{\beta_0} \dots p_{l-1}^{\beta_{l-1}} \ a_2 
		\end{equation}
		with $ \alpha_r, \beta_r < s_r $ and 
		$p_r \nmid  a_1$ and $p_r \nmid a_2$ for each $ r\in [l]$. 
%Note that if $ q \equiv 2 \pmod{4} $, then $ p_c^{s_c} = %2 $  for some $ c \in \{1,2, \dots , l\} $, but in that %case $ n \not \equiv 2 \pmod{4} $ implies that $ j_c %\geq 2 $. 
Now for any $ c \in [l]$ we have the following two cases:
		
		\textbf{Case I:} When $ \alpha_c + \beta_c < s_c $. 
		
		In this case, $i_c+j_c$ is the largest power of 
		$q_c$ which divides $mn$, and thus $i_c+j_c$ is the largest power of $q_c$ which divides $k$. 
		By Theorem \ref{thm:SOFRS.primepowers} and 
		%Lemma \ref{lem:SOFRS_qib1b2}, 
		Corollary \ref{cor:blowitup},
		if $(i_c,j_c,q_c)\neq (1,1,2)$, 
		there exists an array of type
		$I_k(q_c^{i_c}p_c^{\alpha_c},  q_c^{j_c} p_c^{\beta_c}; q_c)$, 
		where $k=i_c+j_c$. 
		However if $(i_c,j_c,q_c)=(1,1,2)$, then 
		$s_c=1$, $\alpha_c=\beta_c=0$ and $2^3$ does not divide $mn$, contradicting $k\geq 3$. 
		%set of $ i_c + j_c $ SOFRS of type $ %FR(q_c^{i_c}p_c^{\alpha_c}, \ q_c^{j_c} %p_c^{\beta_c}; \ q_c^{j_c-1}p_c^{\beta_c}, %q_c^{i_c-1}p_c^{\alpha_c}) $. 
		
		\textbf{Case II:} When	$ \alpha_c + \beta_c \geq s_c$. 
		
		\sloppy Since $ \alpha_c, \beta_c < s_c $, this implies $ \alpha_c + \beta_c < 2s_c $ and therefore $i_c + j_c +1$ is the largest power of 
		$q_c$ which divides $k$. By combining Theorem \ref{thm:SOFRS.primepowers} and Theorem \ref{thm:SOFRS.i+j+1} we obtain an array of type 
		$I_k(q_c^{i_c}p_c^{\alpha_c},  q_c^{j_c} p_c^{\beta_c}; q_c)$ where $k= i_c +j_c +1$. 
		
		%$ i_c +j_c +1 = k $ SOFRS of type $ %FR(q_c^{i_c}p_c^{\alpha_c}, \ q_c^{j_c} %p_c^{\beta_c}; \ q_c^{j_c-1}p_c^{\beta_c}, %q_c^{i_c-1}p_c^{\alpha_c})$. \\
		
		Thus in both cases for each $c\in [l]$ we obtain an array of type  
		$I_k(q_c^{i_c}p_c^{\alpha_c}, q_c^{j_c} p_c^{\beta_c}; q_c)$
		%\ q_c^{j_c-1}p_c^{\beta_c}, %q_c^{i_c-1}p_c^{\alpha_c}) $ 
		and by taking their Kronecker product (see Corollary \ref{cor:kroneckerproduct.of.m.sofs}), we can construct an array of type 
		$I_k(\frac{m}{a_1}, \frac{n}{a_2}; q)$
		%\frac{n}{qa_2}, \frac{m}{qa_1}) $, 
		where $ a_1 $ and $ a_2 $ are defined in equation (\ref{eq:SOFRS.mn}). Finally, by applying Corollary \ref{cor:blowup} we 
		obtain an array of type $I_k(m,n; q)$, which completes the proof. 
	\end{proof}
		
%Theorems \ref{thm:SOFRS.pairs.existence} 
The previous section and Theorem \ref{thm:complete.SOFRS.mod4} together imply Theorem \ref{thm:mainresult}.

%\bibliographystyle{abbrv}
%\bibliography{References}

\end{document}